\documentclass[12pt,intlimits,reqno]{amsart}
\usepackage{amssymb}
\usepackage[T1]{fontenc}
\usepackage[utf8]{inputenc}
\usepackage{enumerate}
\usepackage{verbatim}
\usepackage{bbm}
\usepackage{antpolt}
\usepackage{amsaddr}
\newtheorem{thm}{Theorem}[section]

\newtheorem{prop}[thm]{Proposition}
\theoremstyle{definition}
\newtheorem{example}[thm]{Example}
\newtheorem{defn}[thm]{Definition}
\theoremstyle{remark}
\newtheorem{rem}[thm]{Remark}

\numberwithin{equation}{section}

\newcommand{\be}{\begin{equation}}
\newcommand{\ee}{\end{equation}}

\renewcommand{\H}{\mathcal{H}}
\newcommand{\R}{\mathbb R}

\newcommand{\1}{\mathbbm 1}

\DeclareMathOperator{\supp}{supp}

\newcommand{\pb}{\{\,\cdot\,,\,\cdot\,\}}

\usepackage{skull}
\newcommand{\todo}[1]
  {\vspace{5 mm}\par \noindent \marginpar{{\LARGE$\skull$}} \framebox{\begin%!!!
  {minipage}[c]{0.95 \textwidth} \tt #1
\end{minipage}}\vspace{5 mm}\par}

\begin{document}
   \title{Queer Poisson brackets}
   \author{Daniel Beltita}
\email{beltita@gmail.com}
\address{\small Institute of Mathematics ``S. Stoilow'' of the Romanian Academy, 21 Calea Grivitei Street, 010702 Bucharest, Romania}
\author{Tomasz Goli\'nski}
\email{tomaszg@math.uwb.edu.pl}
\address{\small University in Bia{\l}ystok, Institute of Mathematics, Cio{\l}kowskiego 1M, 15-245~Bia{\l}ystok, Poland}
\author{Alice-Barbara Tumpach}
\email{Barbara.Tumpach@math.univ-lille1.fr}
\address{\small Universit\'e de Lille, Laboratoire Painlev\'e, CNRS U.M.R. 8524, 59 655 Villeneuve d'Ascq Cedex, France}

\begin{abstract}
We give a method to construct Poisson brackets $\pb$ on Banach manifolds~$M$, for which the value of $\{f,g\}$ at some point $m\in M$ may depend on higher order derivatives of the smooth functions $f,g\colon M\to{\mathbb R}$, and not only on the first-order derivatives, as it is the case on all finite-dimensional manifolds. We discuss specific examples in this connection, as well as the impact on the earlier research on Poisson geometry of Banach manifolds. 
Those brackets are counterexamples to 
the claim that the Leibniz property for any Poisson bracket on a Banach manifold would imply the existence of a Poisson tensor for that bracket.
\end{abstract}

\maketitle
%  \tableofcontents
 
\section{Introduction}

The Poisson brackets in infinite-dimensional setting have played for a long time a significant role in various areas of mathematics including mechanics (both classical and quantum) and integrable systems theory (see e.g. \cite{fadeev,bona,ratiu-mta3,marsden-chernoff}). However the rigorous approach to the notion of Poisson manifold in the context of Banach space is relatively recent (see \cite{OR}). It is known that the Poisson brackets on infinite-dimensional manifolds lack some of the properties known from the finite-dimensional case.
It was shown for instance in \cite{OR} that the existence of Hamiltonian vector fields requires an additional condition on the Poisson tensor in the case of manifolds modelled on a non-reflexive Banach space (i.e. a Banach space $E$ that is not canonically isomorphic to its second dual $E\varsubsetneq E^{**}$, where $E^*$ denotes the topological dual of a Banach space). Another example of a new behaviour can be found in \cite{dito} --- a Poisson bracket defined only on a certain space of smooth functions might lead to an unbounded Poisson tensor. 
% Other peculiar aspects of brackets in infinite dimensions follow from the discussion in \cite{pelletier}. 
%Moreover for manifolds which are not smoothly regular (i.e. without smooth bump functions), Poisson bracket need not be local although as far as we know a counterexample is not known yet, see \cite{pelletier}.
Moreover on some manifolds, Poisson brackets need not be local although as far as we know a counterexample is not known yet, see a related discussion in \cite{pelletier}.

The aim of this paper is to prove by example still another phenomenon that is specific to Poisson geometry on an infinite dimensional manifold $M$, namely the existence of Poisson brackets of higher order. That is, Leibniz property does not ensure that the bracket depends only on the first-order derivatives of functions.
The constructed Poisson brackets serve as a counterexample to the statements given in the literature (see \cite{OR} or subsequently \cite{ida2011}), where it was claimed that
the existence of a Poisson tensor $\Pi$ follows from Leibniz property and skew symmetry of the Poisson bracket $\{\cdot,\cdot\}$,
%demonstrate that, in the infinite-dimensional setting, contrary to the arguments given in the literature (see e.g. \cite{OR}), the Leibniz property does not imply that the value of the Poisson bracket depends only
%on the differentials of functions, 
%which in turn together with skew-symmetry would imply 
in particular for every $m\in M$ one could find a bounded bilinear functional $\Pi_m\colon T^*_mM\times T^*_mM\to\R$ satisfying
%as a section of the bundle $\bigwedge^2 T^{**}M$
%satisfying the condition
$$ %\label{P-tensor} 
\{f, g\}(m) = \Pi_m(f'_m,g'_m)
%\Pi_m (df(m), dg(m))
$$
% where $\bigwedge^2 T^{**}M$ is the bundle of skew-symmetric bilinear functions on the fibers of $T^*M$.
where $f'_m,g'_m\in T^*_mM$ are the differentials of $f,g\in C^\infty(M)$ at point $m\in M$. 
There is a  related fact in 
%This claim have roots in 
\cite[Thm. 4.2.16]{ratiu-mta3}, but we show that it is not applicable here (see Proposition~\ref{no_extension}).

We prove that there exist 
%the existence of 
Poisson brackets not given by Poisson tensors on the family of Banach sequence spaces $l^p$ for $1\leq p\leq2$ and present an explicit 
example for $p=2$.
Such Poisson brackets do not allow to introduce the dynamics by Hamilton equations in the usual way, thus from the point of view of applications in physics
one should explicitly assume the existence of Poisson tensor in the definition of a Poisson Banach manifold.
% that is slightly different from the one in \cite{OR}.

%Note that in the case of non-local Poisson brackets on not smoothly regular manifold $M$, the usual construction of Poisson tensor does not work in general even for Poisson brackets of order 1 as the differentials of globally defined functions at a given point $m$ might not span the whole $T^*_mM$.
%\todo{check}

In section \ref{section1} we investigate "queer operational tangent vectors", that is 
derivations on spaces of smooth functions on the manifold which are differential operators of order higher than 1. This notion was introduced with several results on their existence (including the examples on the Hilbert space) in \cite{michor}. We explore the case of queer vectors of order 2 on the family of Banach sequence spaces $l^p$ for $1\leq p<\infty$.

%The example of Poisson bracket not given by a Poisson tensor is obtained using the ``queer operational tangent vectors'' described in \cite{michor}. Those vectors are given as derivations on the space of functions but cannot be obtained as classes of equivalence of curves on the manifold.
% As shown there, there exist derivations of order higher than 1 on $C^\infty(\H)$, where $\H$ is an infinite-dimensional separable Hilbert space. 
%We give a Banach space version of that construction in section~\ref{section1} along with a concrete realization of it. 
% In section \ref{section1} 

Section \ref{section2} contains our main result, which shows a way to construct higher order Poisson brackets out of queer vector fields, and we illustrate the general result by a specific example on the Hilbert space. We conclude the paper with a version of the definition of Banach Poisson manifold which clarifies the one introduced in \cite{OR}. Some discussion on the problem of localization of Poisson bracket is also included.

% Before we continue with the discussion of our counterexamples of Poisson brackets, it is worth pointing out the reference 
% \cite[Thm. 4.2.16]{ratiu-mta3}, where
% the authors prove that for manifolds $M$ modelled on Banach spaces with norm smooth away from the origin, a certain space of derivations is isomorphic to the vector space of vector fields on $M$. However that result is not applicable in the study of Poisson brackets on Banach manifolds.
% Indeed in this reference, a derivation is a collection of maps $C^\infty(M,F)\to C^\infty(M,F)$ for all Banach spaces $F$ (see \cite[Prop. 4.2.9]{ratiu-mta3}), like the Lie derivative for instance. 
% %In particular, the proof of \cite[Thm. 4.2.16]{ratiu-mta3} relies on the fact that the derivation can be applied to a chart function. 
% As we will see in Proposition~\ref{no_extension}, a derivation on the space of real smooth functions 
% on an infinite-dimensional manifold may not extend to a map acting on smooth functions with values in an arbitrary Banach space in such a way that the Leibniz rule still would be satisfied.

% In section~\ref{section2}, we use this queer operational vector field to  construct a Poisson bracket which cannot be expressed by formula \eqref{P-tensor-point} (see Example~\ref{queer_Poisson}).

All Banach and Hilbert spaces considered in this paper are real. By manifold we will always mean a smooth real manifold modelled on a Banach space.

\section{Queer operational vector fields}\label{section1}

% \subsection{Kriegl--Michor example of operational tangent vector of order 2}

There are two major approaches to tangent vectors, namely the kinematic one and the operational one. These approaches lead to the same notion for finite-dimensional manifolds, but this is no longer the case in infinite dimensions. 
A kinematic tangent vector to a Banach manifold $M$ at a point $m\in M$ is an equivalence class of curves passing through that point (for precise definition see e.g.~\cite{ratiu-mta3}). 
On the other hand, an operational tangent vector is defined as a derivation acting in the space of germs of functions (see \cite{michor}, \cite{pelletier}).

For any $m\in M$ consider the set of all functions $f\colon U\to \R$ defined on an open neighborhood $U$ of $m$. One defines an equivalence
relation in that set in the following way: two functions $f_1\colon U_1\to\R$ and $f_2\colon U_2\to\R$ are equivalent if there exists an open neighborhood
$U\subset U_1\cap U_2$ of $m$ for which the restrictions of $f_1$ and $f_2$ to $U$ coincide. 
Any equivalence class %$[f]$ 
defined in this way is called a germ at the point $m\in M$. We denote the set of germs of all smooth functions at $m$ by $C^\infty_m(M)$. 
We note that the value and the derivatives of germs at $m\in M$ (that is, jets of germs) are well defined.
%We will drop the brackets in the notation of germ when there is no risk of confusion.
% Recall that the germ of a function $f:U\to\R$, $U\subset M$, at the point $m\in U$ is the equivalence class of functions which coincide with $f$ on some open neighborhood of $m$. 

We denote by  $L_k(T_m M;\R)$ the Banach space of bounded $k$-linear functionals on $T_mM$ with values in $\R$
and let $f^{(k)}_m\in L_k(T_m M;\R)$ be the $k$-th differential at the point $m\in M$ of a germ or a function. 
%function $f$ at $m\in M$.

\medskip

\begin{defn}
An \textbf{operational tangent vector} at point $m\in M$ is a linear map $\delta: C^\infty_m(M)\to \R$ satisfying  Leibniz rule~:
\be\label{leibniz} \delta (fg) = \delta f\; g(m) + f(m)\; \delta g.\ee 
For any open subset $U\subseteq M$ with $m\in U$ there is a canonical map $C^\infty(U)\to C^\infty_m(M)$ that takes every function on $U$ to its germ at $m$, 
hence one has a canonical pull-back of $\delta$ to $C^\infty(U)$, also denoted by $\delta$. 

An \textbf{operational vector field} on $M$ is a collection of maps $\delta_U: C^\infty(U)\to C^\infty(U)$ for each open set $U\subset M$, compatible with restrictions to open subsets and defining an operational tangent vector $\delta_m$ at every $m\in M$.
% i.e. a derivation of the algebra $C^\infty(M)$.
\end{defn}

\medskip

%Let us denote by  $L_k(T_m M;\R)$ the space of bounded $k$-linear maps on $T_mM$ with values in $\R$ andlet $f^{(k)}_m\in L_k(T_m M;\R)$ be the $k$-th differential of a germ $f\in C^\infty_m(M)$. 
%function $f$ at $m\in M$.

\begin{defn}
The operational tangent vector $\delta$ is of {\bf order $n$} if it can be expressed in the form
\be \label{order} \delta f = \sum_{k=1}^n 
%\ell_k(d^k(f)(m)),
\ell_k(f^{(k)}_m),
\ee
where $\ell_k:L_k(T_m M;\R)\to \R$ are continuous and linear. Moreover we require that $\ell_n$ does not vanish identically on the subspace of symmetric $n$-linear maps in $L_k(T_m M;\R)$. 
Otherwise the order of $\delta$ is infinite.
%If only one term is present in the sum \eqref{order}, $\delta$ is called homogeneous. 
The operational tangent vectors of order at least $2$ are called {\bf queer}.

The operational vector field $\delta$ is of {\bf order at most $n$} if there exists a family of smooth sections 
$\ell_k$ of the bundle $\bigsqcup\limits_{m\in M}(L_k(T_mM;\R))^*$ satisfying \eqref{order} at each $m\in M$.
\end{defn}

\medskip

The Leibniz rule \eqref{leibniz} satisfied by $\delta$ implies certain algebraic conditions on functionals $\ell_k$, see \cite[28.2]{michor}.

By definition, operational tangent vectors of order $n$ depend only on the $n^{th}$ jet of functions. The existence of infinite order operational tangent vectors is an open problem as far as we know.

\medskip

\begin{rem}\label{rem:order1}
Any kinematic tangent vector defines an operational tangent vector of order 1. On the other hand in the case of manifolds modelled on non-reflexive Banach spaces, operational tangent vectors of order $1$ are given by elements of $T^{**}M$ which is larger than the (kinematic) tangent bundle $TM$.
Thus in the case of Banach manifolds (even the ones having a global chart, as for instance Banach spaces), the notions of kinematic tangent vector and operational tangent vector do not coincide in general.
\end{rem}

\medskip

There are examples of Banach spaces possessing queer operational tangent vectors even in the reflexive case. 
A construction of %homogeneous 
second order operational tangent vectors on Hilbert spaces was given in \cite{michor} and we explore it below for a class of Banach spaces. 
Let $E$ be a Banach space and 
%let's consider a non-zero functional $\ell:L_2(E;\R)\to \R$. 
consider the natural inclusion of $E^*\times E^*$ into $L_2(E;\R)$ by~:
\begin{equation}\label{inclusion_rank_one}
\begin{array}{lll}
E^*\times E^*& \rightarrow & L_2(E;\R)\\
(f, g) & \mapsto & \left(f\otimes g~:(v,w)\mapsto f(v)g(w)\right).
\end{array}
\end{equation}
In general (contrary to the finite-dimensional case) the linear span of its image may not be dense. 
 A  functional $\ell\in(L_2(E;\R))^*$ defines an operational tangent vector of order $2$ at any $a\in E$ by
 \be \label{op_ell} 
 \delta_\ell f = 
 \ell(f''_a)
 %\ell(d^2(f)(a))
 \ee
 if and only if it  vanishes on $E^*\times E^*$ regarded as a subspace of $L_2(E;\R)$ via \eqref{inclusion_rank_one}. 
% Leibniz rule then follows directly from vanishing of $\ell$ on $E^*\times E^*$ (see \cite{michor}). 
We also recall here that we can identify $L_2(E;\R)$ with  $L(E;E^*)$.

\medskip

\begin{prop}\label{prop:lp}
There are no operational tangent vectors of the second order on the Banach space $l^p$ of $p$-summable sequences for $2<p<\infty$. On the other hand, if $1\leq p\leq 2$ there are non-trivial operational tangent vectors of the second order. 
\end{prop}

\begin{proof}
The proof of existence of operational tangent vectors of the second order has common idea with \cite[Rem. 28.8]{michor}. Namely
it is equivalent to the existence of a nonzero continuous linear functional $\ell$ that vanishes on  $(l^p)^*\times (l^p)^*$.

According to Pitt's theorem, every map from $l^p$ to $(l^p)^*$ is compact if $2<p<\infty$, see e.g. \cite{pitt1936}, \cite[Thm. 4.23]{ryan2002}, \cite[Prop 6.25]{fabian2001}. Moreover since all $(l^p)^*$ spaces have the approximation property, the closure of linear span of $(l^p)^*\times (l^p)^*$ coincides with the space of compact operators from $l^p$ to $(l^p)^*$ \cite[Ch. 4]{ryan2002}. So, the only continuous functional $\ell$ which would vanish on $E^*\times E^*$ is the zero functional.
Thus there are no non-zero operational tangent vectors of the second order on $l^p$ for $2<p<\infty$. 

In the case $1\leq p\leq 2$, the inclusion map $\iota:l^p\hookrightarrow (l^p)^*$ is not compact, so using Hahn--Banach theorem it is possible to define 
a non-zero functional $\ell$ on $L_2(E;\R)$ that vanishes on the image of the map~\eqref{inclusion_rank_one}. This
implies the existence of non-zero operational tangent vectors of the second order on $l^p$ for $1\leq p\leq 2$.
\end{proof}

\medskip 

In particular for $p=2$ we obtain an operational tangent vector of the second order on the separable Hilbert space $\H$. We will present this case more explicitly.

\medskip 

\begin{example}[concrete queer operational vector on a Hilbert space]\label{ex:ell}$\;$
The Banach space 
$L_2(\H;\R)$ can be identified with the Banach space of bounded operators $L^\infty(\H)$.
This identification maps a bilinear map $B$ to the operator $A$ defined  by 
\begin{equation}\label{bilinear_to_operators}
B(v, w) = \langle Av, w\rangle
\end{equation}
using Riesz theorem.
%\begin{equation}\label{bilinear_to_operators}
%\begin{array}{lll}
%L_2(\H;\R) & \rightarrow & L^\infty(\H_{\mathbb{R}})\\
%B & \mapsto & \left(A~: v \mapsto B(v, \cdot)\right),
%\end{array}
%\end{equation}
% Denote by $(\H_{\R})^*$ the space of real continuous functionals on $\H$.
The closure of the linear span of $\H^*\times \H^*$ considered as a subspace of $L_2(\H;\R)\simeq L^\infty(\H)$ by inclusion~\eqref{inclusion_rank_one} is the ideal of compact operators on $\H$. 
%It is now possible to 
One can now 
obtain the continuous functional $\ell$ with required properties by putting e.g. $\ell(\1)=1$ where $\1$ denotes the identity map, and $\ell(K) = 0$ for any compact operator $K\in L^\infty(\H)$ and extending it to the whole $L^\infty(\H)$ by means of Hahn--Banach theorem.

Let us now demonstrate explicitly that the operational tangent vector $\delta_\ell$ given by \eqref{op_ell} with $\ell$ defined as above is not a kinematic tangent vector. Without loss of generality we fix the point $a=0$.
Taking for example the function 
\be\label{rho}\rho(v)=\langle v, v \rangle\ee
for $v\in\H$, we get $\rho''_v = 2\, \1$, where we have used the identification $L_2(\H;\R) \simeq L^\infty(\H)$ given by \eqref{bilinear_to_operators}. From definition it follows that $\delta_\ell (\rho) = 2$. On the other hand, any kinematic tangent vector to $\H$ at $0$ can be identified with some $w\in\H$ and
$$w\cdot \rho= \langle w,  0 \rangle + \langle 0, w \rangle = 0.$$
Thus $\delta_\ell$ is in fact a queer tangent vector. 
One can extend $\delta_\ell$ to a queer constant operational vector field on $\H$, which we will denote by the same symbol.
% \flushright$\diamond$
\end{example}

\medskip

Let us note that \cite[Thm. 4.2.16]{ratiu-mta3}
states that for manifolds $M$ modelled on Banach spaces with norm smooth away from the origin, a certain space of derivations is isomorphic to the vector space of kinematic vector fields on $M$. 
In this reference, a derivation $\mathbf{D}$ on the Banach manifold~$M$ is a collection of linear maps $C^\infty(M,F)\to C^\infty(M,F)$
for all Banach spaces $F$, such that for any 
$f\in C^\infty(M,F)$, $g\in C^\infty(M,G)$, and any bilinear map 
$B~: F\times G\rightarrow H$, the following Leibniz rule holds
\begin{equation}\label{extended_leibniz}
\mathbf{D}\left(B(f, g)\right) = B(\mathbf{D}f, g) + B(f, \mathbf{D}g),
\end{equation}
where $F$, $G$, and $H$ are Banach spaces.
An example of such a derivation is the Lie derivative. 
Let us show that existence of $\delta_\ell$ in Example \ref{ex:ell} is not a contradiction with this result.
Namely the operational vector field $\delta_\ell$ cannot be extended to a derivation in the sense of \cite{ratiu-mta3}. 

\medskip

\begin{prop}\label{no_extension}
The queer operational vector field $\delta_\ell$ constructed in Example~\ref{ex:ell} cannot be extended to a derivation on all $C^\infty(\H, F)$ spaces, where $F$ is any Banach space.
\end{prop}

\begin{proof}
Let us assume that there exists an extension $\mathbf{D}_\ell$ of $\delta_\ell$.
%Denote by $\H$ an infinite-dimensional Hilbert space. 
Let 
%$F$ be the dual space $\H^*$ of $\H$, $G = \H$, and 
$B$ be the natural duality pairing between $\H^*$ and $\H$. Consider the maps $f~: \H \rightarrow \H^*$, $v \mapsto \langle v, \cdot\rangle$ and $g$ equal to the identity map on $\H$. Then $B(f, g)(v)= \langle v, v\rangle = \rho(v)$, and 
$$
\mathbf{D}_\ell\big(B(f, g)\big)(v) = \delta_\ell(\rho)(v) = \ell(2\,\1) = 2.
$$ 
%The function $\mathbf{D}_\ell\big(B(f, g)\big)$ satisfies $\mathbf{D}_\ell\big(B(f, g)\big)(0) = 2$.
On the other hand, 
$$
B(\mathbf{D}_\ell f, g)(v) + B(f, \mathbf{D}_\ell g)(v) = B(\mathbf{D}_\ell f(v), v) + \langle v, \mathbf{D}_\ell g(v)\rangle.
$$
This expression vanishes for $v=0$, hence \eqref{extended_leibniz} cannot be satisfied for any extension of $\delta_\ell$.
\end{proof}

% Following \cite[Lemma 32.6]{michor} one can show that only operational vector fields of order $1$ can be extended to a collection of maps $D:C^\infty(M,F)\to C^\infty(M,F^{**})$ satisfying \eqref{extended_leibniz}.
% % an $n$-linear analogue of 

\medskip

\begin{prop}
 Let $\delta$ be an operational vector field of finite order on a manifold $M$. Then the set of points at which it is queer is open while the set of points at which it is kinematic is closed in $M$.
\end{prop}

\begin{proof}
Let $n$ be the order of $\delta$. The set of points at which $\delta$ is not queer is the intersection
$\bigcap\limits_{k=2}^n \ell_k^{-1}(0)$ of level sets of zero sections of coefficients $\ell_k:M\to\bigsqcup\limits_{m\in M}(L_k(T_mM;\R))^*$ of $\delta$. 
Since functionals $\ell_k$ are continuous, the above intersection is a closed set. 

The set of points at which $\delta$ is kinematic is 
%the intersection  
$\bigcap\limits_{k=2}^n \ell_k^{-1}(0) \cap \ell_1^{-1}(TM)$, where
we regard
%consider 
$TM$ as a subbundle of $\bigsqcup\limits_{m\in M}(L_1(T_mM;\R))^* =  T^{**}M$. It is straightforward to check that $TM$ is a closed subset of $T^{**}M$ using local trivialization.
\end{proof}

\section{Queer Poisson brackets}\label{section2}

In this section we will construct 
%consider only 
Poisson brackets which are localizable in the sense of the following definition:

%\begin{defn}
%	A {\bf Poisson bracket} on a manifold $M$ is a family of bilinear maps $\pb_U\colon C^\infty(U)\times C^\infty(U)\to C^\infty(U)$ indexed by the open subsets $U\subseteq M$, satisfying 
%	\begin{enumerate}
%		\item[(i)] skew-symmetry: $\{f, g\}_U = -\{g, f\}_U;$
%		\item[(ii)] Jacobi identity: $\big\{\{f, g\}_U, h\big\}_U + \big\{\{g, h\}_U, f\big\}_U + \big\{\{h, f\}_U, g\big\}_U =  0$;
%		\item[(iii)] Leibniz rule: $\{f, gh\}_U = \{f, g\}_Uh + g\{f, h\}_U$
%	\end{enumerate}
%	and moreover compatible with restrictions, i.e. if $U\subseteq V$ and $f,g\in C^\infty(V)$ then
%	$\{f,g\}_V\vert_U=\{f\vert_U,g\vert_U\}_U$.
%	% A bilinear operation $\pb: C^\infty(M)\times C^\infty(M)
%	% \rightarrow C^\infty(M)$ is called a \textbf{Poisson bracket} on a manifold $M$ if it satisfies~:
%\end{defn}

\medskip

\begin{defn}
A {\bf Poisson bracket} on a manifold $M$ is 
a bilinear operation $\pb: C^\infty(M)\times C^\infty(M)
 \rightarrow C^\infty(M)$ satisfying 
 \begin{enumerate}
 	\item[(i)] skew-symmetry: $\{f, g\} = -\{g, f\};$
 	\item[(ii)] Jacobi identity: $\big\{\{f, g\}, h\big\} + \big\{\{g, h\}, f\big\} + \big\{\{h, f\}, g\big\} =  0$;
 	\item[(iii)] Leibniz rule: $\{f, gh\} = \{f, g\}h + g\{f, h\}$; 
 \end{enumerate}
for all $f,g,h\in C^\infty(M)$. 

A Poisson bracket $\pb$ on $M$ is called {\bf localizable} if it has 
a {\bf localization}, that is, 
a family consisting of a Poisson bracket $\pb_U$
on every open subset $U\subseteq M$, which satisfy  
$\pb_M=\pb$ and are compatible with restrictions, i.e., 
if $U\subseteq V$ and $f,g\in C^\infty(V)$ then
$\{f,g\}_V\vert_U=\{f\vert_U,g\vert_U\}_U$. 
%\medskip
If this is the case, then for any function $h\in C^\infty(M)$, 
its corresponding {\bf Hamiltonian vector field} is the operational vector field given by 
\be X_h(f)(m):= \{h\vert_U,f\}_U(m)\ee
for all $f\in C^\infty(U)$ and $m\in U$, for every open subset  $U\subseteq M$.
\end{defn}

\begin{rem}
A version of Peetre's theorem on a Banach space $E$ was proved in \cite{wells}
to the effect that if a linear map $T:C^\infty(E)\to C^\infty(E)$ is local in the sense that
%where $E$ Banach spaces, 
$\supp Tf\subset \supp f$ for all $f\in C^\infty(E)$, 
then $T$ is a differential operator of locally finite order provided that $E$ satisfies the condition of $B^\infty$ smoothness (existence of bump functions with Lipschitz property for all derivatives). This condition is satisfied e.g. for Hilbert spaces, but not for
the Banach space of real sequences that are convergent to zero. 
%To our knowledge there is neither a proof of this theorem for other Banach spaces nor a counterexample.

From compatibility with restrictions it follows that operational vector fields (including Hamiltonian vector fields) are local in this sense. Thus in the case of $B^\infty$ smooth Banach spaces they are differential operators of locally finite order.
\end{rem}

\medskip

In the following we denote by $\bigwedge^2 T^{**}M$ the bundle of skew-symmetric bilinear functions on the fibers of cotangent bundle $T^*M$ of 
a Banach manifold~$M$.

\medskip

\begin{defn}
A localizable Poisson bracket $\pb$ on $M$ is of {\bf order one} at $m\in M$ if 
there exists a skew-symmetric bounded bilinear functional $\Pi_m\colon T_m^{*}M\times T_m^{*}M\to{\mathbb R}$ with 
\be \label{P-tensor-point} 
\{f, g\}_U(m) = 
\Pi_m(f'_m,g'_m)
%\Pi_m (df(m), dg(m))
\ee
open neighborhoods $U$ of $m$ and all $f,g\in C^\infty(U)$. 
%for all $f,g\in C^\infty(M)$. 
%smooth functions $f,g$ defined on a neighborhood of $m$. 
Otherwise we say that $\pb$ is {\bf queer} at $m\in M$. 

If there exists a smooth section $\Pi$ of the bundle $\bigwedge^2 T^{**}M$ satisfying~\eqref{P-tensor-point} 
at every point $m\in M$, then we say that $\Pi$ is the {\bf Poisson tensor} of the Poisson bracket $\pb$. 
\end{defn}

\medskip

\begin{rem}\label{7january2018}
	In the above definition, if the Poisson bracket is of order one at some point $m\in M$ then 
	%it is not clear that 
	there exists only one functional $\Pi_m$ satisfying~\eqref{P-tensor-point},  as the differentials of
	locally defined functions at 
	%globally defined functions at 
	a given point $m$ 
	%might not 
	span the whole $T^*_mM$. 
	%However this is certainly the case if for instance Poisson bracket is localizable or $M$ has a global chart. 
\end{rem}

%Note that in the case of smoothly regular manifolds, the existence of bump functions implies that every Poisson bracket is localizable, i.e. the whole family of bilinear maps is uniquelly defined by $\pb_M$. Also conventional Poisson brackets are automatically localizable in this sense.

\medskip

\begin{thm}\label{thm}
Let $\delta_1$ and $\delta_2$ be two commuting operational vector fields on a Banach manifold~$M$, and define
$$\{f_1,f_2\}_U:=(\delta_1)_U(f_1)\; (\delta_2)_U(f_2)-(\delta_2)_U(f_1)\; (\delta_1)_U(f_2),$$
for all $f_1$, $f_2\in C^\infty(U)$, for every open subset $U\subseteq M$. 
Then $\pb:=\pb_M$ is a localizable Poisson bracket with a localization consisting of the brackets~$\pb_U$. 
If moreover $\delta_1$ and $\delta_2$ are linearly independent at some point $m\in M$, then the Poisson bracket $\pb$ is queer at the point $m$ if and only if at least one the operational vector field $\delta_1$ and $\delta_2$ is queer at $m$. 
\end{thm}
% \begin{e-proposition}\label{4Nov2016}
% Consider two derivations $D_1$ and $D_2$ of $C^\infty(M)$,
% %Let $\Ac$ be any commutative associative $\R$-algebra 
% %and $D_1,D_2\in\Der(\Ac)$ with 
% such that $[D_1,D_2]=0$. 
% Define 
% $$\{f_1,f_2\}:=D_1(f_1)\cdot D_2(f_2)-D_2(f_1)\cdot D_1(f_2),$$
% for any $f_1$, $f_2\in C^\infty(M)$.
% Then $\{\cdot,\cdot\}\colon C^\infty(M)\times C^\infty(M)\to C^\infty(M)$ is a Poisson bracket. 
% \end{e-proposition}
% 
\begin{proof}
Bilinearity and skew-symmetry of $\pb$ are obvious. Jacobi identity follows from the commutativity of $\delta_1$ and $\delta_2$ just like in the case of canonical Poisson bracket on $\R^2$. This can also be seen e.g. as the special case $n=2$ of \cite[Prop. 2]{Filippov}. 
The Leibniz rule for $\pb$ follows easily from \eqref{leibniz}. Compatibility with restrictions follows from the definition of operational vector fields.

Now assume that $\delta_1$ and $\delta_2$ are linearly independent at $m\in M$. 
If none of $\delta_1$ and $\delta_2$ is queer at $m$, then it follows by Remark \ref{rem:order1}
that their values at $m$ satisfy $(\delta_1)_m,(\delta_2)_m\in T_m^{**}M$. 
Then \eqref{P-tensor-point} is satisfied if we define 
$\Pi_m\colon T_m^{*}M\times T_m^{*}M\to{\mathbb R}$ by 
$$\Pi_m(\mu,\nu)=(\delta_1)_m(\mu)\;(\delta_2)_m(\nu)-(\delta_2)_m(\mu)\;(\delta_1)_m(\nu)\text{ for all }\mu,\nu\in T_m^*M,$$
hence $\pb$ is not queer at $m\in M$.

Conversely, assume that  $\pb$ is not queer at $m\in M$, hence we have \eqref{P-tensor-point}. 
Since the linear functionals $(\delta_1)_m,(\delta_2)_m: C^\infty_m(M)\to \R$ are linearly independent by hypothesis, 
there exist an open subset $U_1\subseteq M$ with $m\in U_1$ and a function $f_1\in C^\infty(U_1)$   
%$f_1\in C^\infty_m(M)$
%$[f_1]\in C^\infty_m(M)$, 
satisfying 
with  
$(\delta_1)_m(f_1)=0$ 
%$((\delta_1)_{U_1}f_1)(m)=0$ 
and 
$(\delta_2)_m(f_1)\ne0$. 
%$((\delta_2)_{U_1}f_1)(m)\neq0$. 
Then for every open subset $U\subseteq M$ with $m\in U$ and every $f\in C^\infty(U)$ we obtain  
$$\begin{aligned}
\{f_1\vert_{U\cap U_1}, f\vert_{U\cap U_1}\}_{U\cap U_1}(m) 
&= ((\delta_2)_{U\cap U_1}(f_1\vert_{U\cap U_1}))(m)\cdot
((\delta_1)_{U\cap U_1}(f\vert_{U\cap U_1}))(m) \\
&=(\delta_2)_m(f_1)\cdot (\delta_1)_m(f)
\end{aligned}$$ 
hence  by \eqref{P-tensor-point} 
$$(\delta_1)_m(f)
=((\delta_1)_{U\cap U_1}(f\vert_{U\cap U_1}))(m)
%=\delta_1(f)(m)
%=\frac{1}{((\delta_2)_{U_1}f_1)(m)}
=\frac{1}{(\delta_2)_m(f_1)}
\Pi_m((f_1)'_m,f'_m)
%\Pi_m( (df_1)(m),df(m)) 
$$
and this shows that the operational tangent vector $(\delta_1)_m$ has order~1 at $m$.
One can similarly prove that the operational tangent vector $(\delta_2)_m$ has order~1 at $m$ and this completes the proof. 
% $\ker \delta_1\neq\ker\delta_2$
\end{proof}

One can use Theorem~\ref{thm} and Proposition~\ref{prop:lp} to construct queer Poisson brackets on $l^p$ spaces for $1\leq p\leq 2$. Again we will present the case $p=2$ in more detail.

\medskip

\begin{example}[concrete queer Poisson bracket]\label{queer_Poisson}
Now let us take $M = \H \times \R$. Denote points of $M$ as $(v,x)$. As the first operational vector field let us take $\delta_\ell$ from Example \ref{ex:ell} acting in $v$ variable, and for the second --- $\frac\partial{\partial x}$. They commute and thus by Theorem~\ref{thm} define a queer Poisson bracket on $\H\times\R$:
$$\{f,g\}(v,x):= \delta_\ell(v) f(\cdot,x) \frac{\partial g}{\partial x}(v,x) - \frac{\partial f}{\partial x}(v,x) \delta_\ell(v) g(\cdot,x).$$

Note that this Poisson bracket has pathological properties: it does not allow Hamiltonian formalism in the usual sense since its corresponding Hamiltonian vector fields are in general only operational vector fields, e.g. for the function $h(v,x)=-x$ is
$$X_h:=\{h,\cdot\} = \delta_\ell.$$
Obviously it is not a section of $TM$. Since in the constructed example $\delta_\ell$ was a differential operator of the second order, it will not lead to an evolution flow on $M$.
Note that the system of Hamilton equations
% \be \dot f = X_h f\ee
$$\frac{d}{dt} f(v(t),x(t)) = (X_h f)(v(t),x(t))$$
for $f\in C^\infty(M)$ is not even a well posed problem. Namely for the function  $\rho$ given by \eqref{rho} we get
\be \frac{d}{dt} \rho(v(t)) = 2.\ee
Now consider the function $f(v,x)=\langle v, w \rangle$ for a fixed vector $w\in\H$. One sees that $f''=0$ and thus $X_h f = 0$.
Since the vector $w$ was arbitrary, it follows that $\frac{d}{dt} v(t) = 0$. 
% The problem is a consequence of the lack of chain rule type of relation for $D_\ell$.

% \flushright$\diamond$
\end{example}

\medskip

As demonstrated a queer Poisson bracket does not lead to the dynamics in the usual way. However it may be possible to consider the dynamics not on the initial manifold but on some jet bundle or higher (co)-tangent bundle, see e.g. \cite{grabowski15} and references therein.

Taking this into account, from the point of view of applications in physics (including classical mechanics) one should explicitly assume the existence of Poisson tensor in the definition of Poisson Banach manifold. This also ensures the existence of the map $\sharp: T^*M \to T^{**}M$ defined by
\be \sharp(\mu_m) = \Pi_m(\mu_m,\cdot), \qquad \mu_m\in T^*_mM.\ee

\begin{defn}
 A {\bf Banach Poisson manifold} $(M,\pb)$ is a Banach manifold $M$ equipped with a 
 %conventional Poisson bracket 
 localizable Poisson bracket $\pb$ for which there exists a Poisson tensor and the corresponding map $\sharp$ satisfies
 \be \label{sharp}\sharp(T^*M)\subset TM.\ee
\end{defn}

%Note that 
This definition is a clarification of the definition of Banach Poisson manifolds given in \cite[Def. 2.1]{OR}, where the localizability property and the existence of Poisson tensor or $\sharp$ map were not explicitly assumed, but were assumed implicitly. In consequence all Banach Poisson manifolds considered there (including Banach Lie--Poisson spaces) do satisfy the corrected definition. 
%Moreover, the Poisson tensor of every Banach Lie-Poisson space is uniquely determined. (See Remark~\ref{7january2018}.)

The condition \eqref{sharp} on the map $\sharp$ was introduced in \cite{OR} and guarantees that Hamiltonian vector fields are kinematic and
it is equivalent to the bilinear functional $\Pi_m\colon T^*M\times T^*M\to\mathbb{R}$ being separately weak$^*$-continuous. 

\section*{Acknowledgements}
We wish to thank Tirthankar Bhattacharyya and Andreas Kriegl for helpful discussions. The work of the first-named author was supported by a grant of the Romanian National Authority for Scientific Research and Innovation, CNCS--UEFISCDI, project number PN-II-RU-TE-2014-4-0370. 
The work of the third-named author was supported by the Labex CEMPI (ANR-11-LABX-0007-01).

% \bibliographystyle{grzyby}
% \bibliography{../literatura}

\end{document}